\documentclass[10pt, reqno]{amsart}

\usepackage{mathrsfs}
\usepackage{mathtools}
\usepackage{extarrows}
\usepackage{enumerate}
\usepackage{url}
\usepackage{hyperref}

\newtheorem{theorem}{Theorem}
\newtheorem{lemma}[theorem]{Lemma}

\newtheorem*{theorem*}{Theorem}

\theoremstyle{definition}

\numberwithin{theorem}{section}
\numberwithin{equation}{section}

\renewcommand{\P}{\mathbf{P}}

\renewcommand{\epsilon}{\varepsilon}

\renewcommand{\(}{\left(}
\renewcommand{\)}{\right)}

\newcommand{\floor}[1]{\left\lfloor#1\right\rfloor}

\newcommand{\Cay}{\operatorname{Cay}}

\begin{document}

\title{The trivial lower bound for the girth of $S_n$}

\author{Sean Eberhard}
\address{Sean Eberhard, London, UK}
\email{eberhard.math@gmail.com}

\begin{abstract}
Consider the Cayley graph of $S_n$ generated by a random pair of elements $x,y$. Conjecturally, the girth of this graph is $\Omega(n \log n)$ with probability tending to $1$ as $n\to\infty$. We show that it is at least $\Omega(n^{1/3})$.
\end{abstract}

\maketitle

\section{Introduction}

Given a group $G$ and a symmetric set $S\subset G$, the \emph{Cayley graph} $\Cay(G, S)$ is the graph with vertex set $G$, and $g$ joined to $gs$ for each $g\in G$ and $s \in S$. Cayley graphs are particularly interesting for being regular and often of large girth (length of the shortest cycle). Note that loops in $\Cay(G,S)$ are essentially the same as relations among the elements of $S$. In particular, when $S = \{x, x^{-1}, y,y^{-1}\}$, with $x,x^{-1},y,y^{-1}$ distinct, then the girth of $\Cay(G,S)$ is the same as the length of the shortest nontrivial word $w\in F_2$ such that $w(x,y) = 1$.

Let $S_n$ be the symmetric group of degree $n$ and let $S = \{x,x^{-1},y,y^{-1}\}$ with $x,y \in S_n$ chosen uniformly at random. It follows from the basic Moore bound in graph theory that the girth of $\Cay(S_n,S)$ is at most $O(n \log n)$. Conjecturally, this bound is tight. The main claim in this note is the following.

\begin{theorem}\label{main-thm}
With high probability the girth of $\Cay(S_n,S)$ is at least $\Omega(n^{1/3})$.
\end{theorem}

In \cite{ghssv} it is claimed that $\Cay(S_n,S)$ almost surely has girth at least $\Omega(\sqrt{n \log n})$. Unfortunately, there is a hole in the proof (also reproduced in \cite{ellis-linial}). We explain the bug in Section~\ref{sec:bug} for the benefit of the interested reader. Thus, modest though it is, Theorem~\ref{main-thm} appears to be the best known lower bound for this problem.

\section{The proof}

We broadly follow the claimed proof in \cite{ghssv}, weakening the claims where necessary.

As a slight generalization we consider $d$ random generators $\pi_1,\dots,\pi_d \in S_n$, for any fixed $d\geq 2$. Let $S = \{\pi_1^{\pm1}, \dots, \pi_d^{\pm1}\}$. We claim the following.

\begin{theorem}\label{main-thm-d}
With high probability the girth of $\Cay(S_n,S)$ is at least
\[
\Omega\(\(\frac{n}{\log(2d-1)}\)^{1/3}\).
\]
\end{theorem}

Write $a_1,\dots,a_d$ for the generators of $F_d$. If $w\in F_d$ and $\pi_1,\dots,\pi_d\in S_n$ we write (as we have done already) $w(\pi_1,\dots,\pi_d)$ image of $w$ under the substitution $a_1\mapsto \pi_1,\dots, a_d\mapsto \pi_d$. We write $P_{S_n}(w)$ for the probability that $w(\pi_1,\dots, \pi_d)=1$ when $\pi_1,\dots,\pi_d$ are chosen at random from $S_n$.

\begin{lemma}
Suppose that $w\in F_d$ has length $k>0$. Then for any $m < n/k$ we have
\[
  P_{S_n}(w) \leq \(\frac{mk^2}{n-mk}\)^m.
\]
In particular for a constant $c>0$ we have
\[
  P_{S_n}(w) \leq \exp(-cn/k^2).
\]
\end{lemma}
\begin{proof}
Write
\[
  w = w_k \cdots w_1,
\]
where each $w_i \in \{a_1^{\pm 1}, \dots, a_d^{\pm 1}\}$ and $w_i \neq w_{i-1}^{-1}$ for each $i>1$. Let $\pi_1, \dots, \pi_d \in S_n$ be random.

Pick $x_1 \in \{1,\dots,n\}$ arbitrarily, and examine the letter-by-letter trajectory $x_1^0,\dots,x_1^k$ of $x_1$ under $w$, revealing the values of $\pi_1,\dots,\pi_d$ on a need-to-know basis. Explicitly,
\begin{align*}
 x_1^0 &= x_1,\\
 x_1^j &= w_j(\pi_1, \dots, \pi_d) x_1^{j-1} \qquad (0 < j \leq k).
\end{align*}
Note that $w$ fixes $x_1$ if and only if $x_1^k = x_1$, but in fact it's rather more likely that $x_1^0, \dots, x_1^k$ are all distinct. Indeed, conditional on $x_1^0, \dots, x_1^{j-1}$ all being distinct, the value of $w_j(\pi_1, \dots, \pi_d) x_1^{j-1}$ is still unexposed, since the only thing we have revealed about $x_1^{j-1}$ so far is that $w_{j-1}(\pi_1,\dots,\pi_d) x_1^{j-2} = x_1^{j-1}$, and $w_j \neq w_{j-1}^{-1}$. Thus $x_1^j$ will be drawn from a pool of at least $n - j + 1$ points, so
\[
  \P(x_1^j \in \{x_1^0, \dots, x_1^{j-1}\} | x_1^0, \dots, x_1^{j-1}) \leq \frac{j}{n-j+1}.
\]
Thus by a union bound we have
\[
  \P(x_1^0, \dots, x_1^k~\text{not all distinct}) \leq \sum_{j=1}^k \frac{j}{n-j+1} \leq \frac{k^2}{n-k}.
\]

Supposing nevertheless that $x_1^k = x_1$, we may still pick $x_2 \notin \{x_1^0, \dots, x_1^k\}$ and repeat the argument. In fact, suppose we have done this $m-1$ times already, and we have just chosen $x_m \notin \bigcup_{i<m} \{x_i^0, \dots, x_i^k\}$. Define the trajectory $x_m^0, \dots, x_m^k$ as before. Assuming $x_m^0, \dots, x_m^{j-1}$ are all distinct and disjoint from $\bigcup_{i<m} \{x_i^0, \dots, x_i^k\}$, $x_m^j$ will be drawn from a pool of at least $n - (m-1)k - j + 1$ points, so the probability that
\[
  x_m^j \in \{x_m^0, \dots, x_m^{j-1}\} \cup \bigcup_{i<m} \{x_i^0, \dots, x_i^k\}
\]
is at most
\[
  \frac{(m-1)k + j}{n - (m-1)k - j + 1} \leq \frac{mk}{n - mk}.
\]
Thus the probability that the trajectory $x_m^0, \dots, x_m^k$ fails to be injective, or fails to avoid $\bigcup_{i<m} \{x_i^0, \dots, x_i^k\}$, is at most
\[
  \frac{mk^2}{n-mk}.
\]

In order to have $w(\pi_1,\dots,\pi_d) = 1$ we must have $x_i^k = x_i$ for each $i$. Thus we have
\[
  P_{S_n}(w) \leq \prod_{i=1}^m \P(x_i^k = x_i | x_1^k = x_1, \dots, x_{i-1}^k = x_{i-1}) \leq \( \frac{mk^2}{n-mk}\)^m.
\]
Put $m = \floor{n/(4k^2)}$ to get
\[
  P_{S_n}(w) \leq 2^{-n/(4k^2) + 1}.\qedhere
\]
\end{proof}

\begin{proof}[Proof of Theorem~\ref{main-thm-d}]
Let $W_k \subset F_d$ be the set of all words of length at most $k$. With high probability, the elements $\pi_1^{\pm1},\dots,\pi_d^{\pm1}$ are all distinct. Thus $\Cay(S_n,S)$ has girth at most $k$ if and only if there is some nontrivial $w\in W_k$ such that $w(\pi_1,\dots, \pi_d)=1$. The probability of this event is bounded by
\[
  \sum_{\substack{w\in W_k\\ w\neq 1}} P_{S_n}(w) \leq |W_k| \max_{\substack{w\in W_k\\ w\neq 1}} P_{S_n}(w) \leq (2d) \cdot (2d-1)^{k-1} \cdot \exp(-cn/k^2).
\]
If we put
\[
  k = \floor{\(\frac{cn}{2 \log(2d-1)}\)^{1/3}},
\]
then we get
\[
  \P(\operatorname{girth}(\Cay(S_n,S)) \leq k) \leq \exp\(-cn^{1/3} \log(2d-1)^{1/3}\),
\]
so the theorem follows.
\end{proof}

\section{The bug in \cite{ghssv}}\label{sec:bug}

For simplicity take $d=2$. In the previous section we attempted to show that $w(\pi_1,\pi_2)\neq 1$ by trying to find a point $x$ such that the trajectory given by $w(\pi_1,\pi_2)$ acting on $x$ is injective. The proof of \cite[Section~3]{ghssv} attempts to improve on this by just finding a point not returning to its starting point. The proof goes roughly as follows:
\begin{enumerate}
  \item We may assume that $w$ is cyclically reduced, i.e., that $w_1 \neq w_k^{-1}$ (this is true).
  \item Assuming $w$ is cyclically reduced and that $w(\pi_1,\pi_2) = 1$, there must be some index $j>0$ for which both (a) $x_1^j = x_1$ and (b) $w_j \neq w_1^{-1}$ (this is also true).
  \item We can bound the probability of this happening for the first time at step $j$, conditional on the trajectory up to $j-1$, by $1/(n-j+1)$ (this is false).
\end{enumerate}

To see how this goes wrong, consider for example the word
\[
  w = a b a^{-1} b a b^{-1},
\]
and suppose the trajectory of $1$ turns out to be
\[
  1 \xlongrightarrow{b^{-1}} 2 \xlongrightarrow{a} 2 \xlongrightarrow{b} 1 \xlongrightarrow{a^{-1}} 3 \xlongrightarrow{b} 3 \xlongrightarrow{a} 1.
\]
The first step at which the trajectory returns to $1$ by a letter other than $w_1^{-1} = b$ is the last step, but by this point the transition is determined, so we do not have any bound on the conditional probability.

In brief, the method of \cite[Section~3]{ghssv} attempts to short-circuit the injectivity approach by focusing just on the most important event, that of returning to the starting point. Unfortunately it seems rather difficult to control this event without demanding complete injectivity of the trajectory.

Still, it seems reasonable that some argument of this form should be able to prove a lower bound of the claimed form $\Omega(\sqrt{n \log n})$.

\bibliography{cayley-girth}
\bibliographystyle{plain}

\end{document}